\documentclass{article}
\usepackage{amssymb,amsmath,amsthm,amscd,verbatim}
\usepackage{hyperref}
\usepackage{mathrsfs}
\usepackage{graphicx}

\newtheorem{thm}{Theorem}[section]
\newtheorem{lem}[thm]{Lemma}
\newtheorem{pro}[thm]{Proposition}

\theoremstyle{definition}
\newtheorem{df}[thm]{Definition}
\theoremstyle{remark}
\newtheorem{rem}{Remark}

\newcommand{\restrict}{\upharpoonright}

\newcommand{\N}{\mathbb{N}}
\renewcommand{\P}{\mathbb{P}}
\newcommand{\inter}{\cap}
\newcommand{\Inter}{\bigcap}
\renewcommand{\to}{\rightarrow}

\newcommand{\eps}{\varepsilon}
\newcommand{\nil}{\varnothing}
\newcommand{\ms}{\mathscr}
\newcommand{\mf}{\mathfrak}
\newcommand{\mc}{\mathcal}
\newcommand{\Union}{\bigcup}
\newcommand{\union}{\cup}
\newcommand{\Iff}{\Leftrightarrow}
\newcommand{\Implies}{\Rightarrow}

\begin{document}
\title{Effective dimension of points visited by Brownian motion}

\author{Bj{\o}rn Kjos-Hanssen\footnote{Department of Mathematics, University of Hawai{\textquoteleft}i at M{\=a}noa, Honolulu, HI 96822; bjoern@math.hawaii.edu. Partially supported as co-PI by NSF grant DMS-0652669.} \\
Anil Nerode\footnote{Department of Mathematics, Cornell University, Ithaca, NY 14853; anil@math.cornell.edu}}

\maketitle

\begin{abstract}
We consider the individual points on a Martin-L\"of random path of Brownian motion. We show (1) that Khintchine's law of the iterated logarithm holds at almost all points; and (2) there exist points (besides the trivial example of the origin) having effective dimension $<1$. The proof of (1) shows that for almost all times $t$, the path $f$ is Martin-L\"of random relative to $t$ and so the effective dimension of $(t,f(t))$ is 2.
Keywords: Brownian motion, algorithmic randomness, effective randomness
\end{abstract}

\section{Introduction}

Algorithmic randomness for Brownian motion was introduced by Asarin and Pokrovskii. They defined what they called, according to the English translation  \cite{AP}, \emph{truly random} continuous functions. Fouch\'e \cite{F} called these functions \emph{complex oscillations}.

In this article we answer a question of Fouch\'e (see \cite{F2}) by showing that for each complex oscillation, Khintchine's law of the iterated logarithm holds at almost every point. To that end, in Section \ref{mrbrown} we borrow a construction from the proof of the Wiener-Carath\'eodory measure algebra isomorphism theorem. For the full statement of this theorem, the reader may consult for example Royden \cite{R}, Theorem 15.3.4; we shall not need it.

We believe our method based on this isomorphism theorem can be used to yield other results than the one presented here. Namely, algorithmic randomness for the unit interval $[0,1]$ has been studied more extensively than algorithmic randomness for the space $C[0,1]$ of continuous functions, and the isomorphism theorem allows a transfer of some results. For a general introduction to algorithmic randomness on $[0,1]$, the reader may consult \cite{LV}.

In algorithmic randomness and in computability theory generally, the Turing oracles considered are usually drawn from the space $2^\N$ of infinite binary sequences. Since all non-computable real numbers have a unique binary expansion, it makes no difference if oracles are drawn from the unit interval $[0,1]$ instead.

\begin{df} Suppose $\Omega$ is a set, $\ms F=\{T_i:i\in\N\}$ a countable Boolean algebra of subsets of $\Omega$, $\mu$ a probability measure on the $\sigma$-algebra generated by $\ms F$. Let $t\in [0,1]$. Suppose $\phi:\N^2\to\N$ is a total function Turing reducible to $t$.
The sequence $U_n=\bigcup_m T_{\phi(n,m)}$, $n\in\N$ is called a $t$-uniform sequence of $\Sigma^t_1(\ms F)$ sets.
A $t$-effective $\ms F$-null set is a set $A\subseteq\Omega$ such that for some such $\phi$,
\begin{enumerate}
\item $A\subset\bigcap_n U_n$, and
\item $\mu U_n$ goes effectively to $0$ as $n\to\infty$. That is, there is a computable function $\psi$ such that whenever $n\ge \psi(k)$, we have $\mu U_n\le 2^{-k}$.
\end{enumerate}
\end{df}

We review the Wiener probability measure $W$ on $\Omega=C[0,1]$. It is such that for $\omega\in \Omega$, and $t_0<t_1<\cdots<t_n$, the values of $\omega(t_0)$ and $\omega(t_{i+1}-t_i)$ are independent random variables.  
Moreover, the probability that $\omega(s+t)-\omega(s)\in A$, where $A$ is some set of reals, is $\int_A (2\pi t)^{-1/2}\exp(-x^2/2t)dx$. This says that $\omega(t)$ is normally distributed with standard deviation $\sqrt{t}$ (variance $t$) and mean $0$. Informally, a sufficiently random member of $\Omega$ with respect to $W$ is called a path of Brownian motion.

The precise definition of complex oscillations is immaterial to the present paper, but we include it for completeness. The idea is to mimic the classical characterization of Brownian motion as a limit of random walks with finer and finer increments (Donsker's Invariance Principle).

\begin{df}\label{co}
For $n\ge 1$, we write $C_n$ for the class of continuous functions on $[0,1]$ that vanish at $0$ and are linear with slope $\pm\sqrt{n}$ on the intervals $[(i-1)/n,i/n]$, $i=1,\ldots,n$.

To every $x\in C_n$ one can associate a binary string in $\{1,-1\}^*$,  $a_1\cdots a_n$, of length $n$ by setting $a_i=1$ or $a_i=-1$ according to whether $x$ increases or decreases on the interval $[(i-1)/n,i/n]$. We call the word $a_1\cdots a_n$ the \emph{code} of $x$ and denote it by $c(x)$.

A sequence $\{x_n\}_{n\in\N}$ in $C[0,1]$ is \emph{complex} if $x_n\in C_n$ for each $n$ and there is some constant $d\in\N$ such that $K(c(x_n))\ge n-d$ for all $n$, where $K$ denotes prefix-free Kolmogorov complexity.

A function $x\in C[0,1]$ is a \emph{complex oscillation} if there is a complex sequence $\{x_n\}_{n\in\N}$ such that $x_n-x$ converges effectively to $0$ as $n\to\infty$, in the uniform norm.
\end{df}

A number $t\in [0,1]$ is a dyadic rational if it is of the form $\frac{p}{2^n}$, for $p,n\in\N$; otherwise, $t$ is called a dyadic irrational. 

In the following, $\text{closure}(G)$ is the closure of $G$, $G^{\text{co}}$ is the complement of $G$, and $O_\eps(G)$ is the open $\eps$-ball around $G$.

\begin{df}[Fouch\'e \cite{F}] \label{egs}
A sequence $\ms F_0=(F_i:i\in\N)$ of Borel subsets of $\Omega$ is a $t$-effective generating sequence if
\begin{enumerate}
\item[(1)] for $F\in\ms F_0$ and $\eps>0$, if 
$$G\in \{\{O_\eps(F), O_\eps(F^{\text{co}}), F, F^{\text{co}}\},$$ 
then $W(\text{closure}(G))=W(G)$;
\item[(2)] there is a $t$-effective procedure that yields, for each sequence $0\le i_1<\cdots<i_n\in\N$ and $k\in\N$, a dyadic rational number $\beta_k$ such that $|W(\Inter_{1\le k\le n} F_{i_k})-\beta_k|<2^{-k}$; and
\item[(3)] for $n,i\in\N$, for rational numbers $\eps>0$ and for $x\in C_n$, both the relations $x\in O_\eps(F_i)$ and $x\in O_\eps(F^{\text{co}}_i)$ are $t$-recursive in $x, \eps, i$ and $n$.
\end{enumerate}

If there exists a $t$ such that $\ms F_0$ is a $t$-effective generating sequence, then $\ms F_0$ is called a \emph{generating sequence}. The algebra it generates is similarly called a \emph{generated algebra}. A \emph{$t$-effectively generated algebra} is the Boolean algebra generated from an $t$-effective generating sequence.
If $F$ is a generated algebra and $\omega$ belongs to no $t$-effective $\ms F$-null set, then we say that $\omega$ is $t$-$\ms F$-random or $\ms F$-random relative to $t$. If $t$ is computable then we may omit mention of $t$.
A set $A\subset C[0,1]$ is of \emph{$t$-constructive measure 0} if, for some $t$-effectively generated algebra $\ms F$, $A$ is a $t$-effective $\ms F$-null set.
\end{df}

\begin{thm}[Fouch\'e \cite{F}; see also \cite{F1}]
No complex oscillation belongs to any set of constructive measure 0.
\end{thm}

Let $LIL(\omega,t)$ be the statement that

$$\limsup_{h\to 0}\frac{|\omega(t+h)-\omega(t)|}{\sqrt{2 |h| \log\log (1/|h|)}}=1.$$

\noindent Thus $LIL(\omega,t)$ says that Khintchine's Law of the Iterated Logarithm holds for $\omega$ at $t$.

\begin{thm}[following Fouch\'e \cite{F2}]\label{platformA}
If $t\in [0,1]$, and $f$ is $t$-$\ms F$-random for each $t$-effectively generated algebra $\ms F$, then the Law of the Iterated Logarithm holds for $f$ at $t$.
\end{thm}

The proof is a straightforward relativization to $t$ of Fouch\'e's argument (which covers the case where $t$ is computable).

\section{Isomorphism Theorem}\label{mrbrown}

Let $\ms A_0$ be a generating sequence.
 Write $\ms A_0=\{A_n\}_{n\in\N}$.

\begin{itemize}
\item Let $\mathfrak A_n$ be the Boolean algebra generated by $\{A_1,\ldots,A_n\}$.
\item Let $\ms A=\mathfrak A_\infty=\bigcup_n\mathfrak A_n$, the Boolean algebra generated by $\ms A_0$.
\item Let $\mf I$ be the Boolean algebra of finite unions of half-open intervals $[a,b)$ in $[0,1)$.
\end{itemize}

A \emph{Boolean measure algebra homomorphism} is a map that preserves measure, unions, and complements.

 \begin{thm}[Wiener, Carath\'eodory]\label{XYZ}
There is a Boolean measure algebra homomorphism $\Phi:\ms A\to\mathfrak I$.
\end{thm}
  \begin{proof}
In this proof we will denote Wiener measure $W$ by $\mu$. We first consider the case $n=1$. Since $\mf A_1=\{\emptyset, A_1, A^{\text{co}}_1, \Omega\}$ and $\mu A_1+\mu A^{\text{co}}_1=\mu\Omega=1$, we let $\Phi(A_1)=[0,\mu A_1)$, $\Phi(A^{\text{co}}_1)=[\mu A_1,1)$, $\Phi(\emptyset)=\emptyset$, and $\Phi(\Omega)=[0,1)$. Then $\Phi$ is clearly a Boolean measure algebra homomorphism from $\mf A_1$ into $\mf I$.

Suppose now that $\Phi$ has been defined on $\mf A_{n-1}$ so that it is a Boolean measure algebra homomorphism from $\mf A_{n-1}$ onto the algebra generated by $k\in\N$ many half open intervals $[x_0,x_1), [x_1,x_2),\ldots,[x_{k-1},x_k)$, where $x_0=0$ and $x_k=1$.

We wish to extend the mapping $\Phi$ to $\mf A_n$. Let $B_i$ be the set in $\mf A_{n-1}$ which is mapped onto the interval $[x_j,x_{j+1})$, for $j<k$. Then $\mf A_{n-1}$ consists of all finite unions of the sets $B_j$, $j<k$, and $\mf A_n$ consists of all finite unions from the $2k$ sets $A_n\inter B_j$, $A^{\text{co}}_n\inter B_j$, $j<k$. Let

 $$\Phi(A_n\inter B_j)=[x_j,x_j+\mu(A_n\inter B_j))$$
 $$\Phi(A^{\text{co}}_n\inter B_j)=[x_j+\mu(A_n\inter B_j),x_{j+1})$$

\noindent This might define $\Phi$ of some sets to be of the form $[x_j,x_j)=\emptyset$.

Clearly $\Phi$ as so defined preserves Lebesgue measure on $[0,1)$. Moreover 
$\Phi(A_n\inter B_j)\cup \Phi(A^{\text{co}}_n\inter B_j)=[x_j,x_{j+1})=\Phi(B_j)$, and $\mu(A_n\inter B_j)+\mu(A^{\text{co}}_n\inter B_j)=\mu(B_j)=x_{j+1}-x_j$. From this it follows that we can extend $\Phi$ to all of $\mf A_n$ so that it is a Boolean measure algebra homomorphism. Since $\mf A_\infty=\bigcup_n\mf A_n$, we have thus defined $\Phi$ on all of $\mf A_\infty$.
 \end{proof}

\begin{rem}\label{referee} The function $\Phi$ is effective in the following sense: if $\ms F=\{T_k:k\in\N\}$ is a $t$-effectively generated algebra, then the measure of $\Phi(T_k)$ can be computed $t$-effectively, uniformly in $k$.
\end{rem}

\begin{lem}\label{245}
Suppose $\mc I_n=(a_n,b_n)$, $n\in\N$, is a sequence of open intervals with $(a_{n+1},b_{n+1})\subseteq (a_n,b_n)$. Suppose $\bigcap_n (a_n,b_n)=\emptyset$. Then either $\{a_n\}_{n\in\N}$ or $\{b_n\}_{n\in\N}$ is an eventually constant sequence.
\end{lem}
The proof is routine. The set of Martin-L\"of real numbers in $[0,1]$ is denoted $\text{RAND}$, and relativized to $t$, $\text{RAND}^t$.

An effectively generated algebra $\ms F=\{T_k:k\in\N\}$ is \emph{non-atomic} if for any $b:\N\to \{0,1\}$, we have $W(\Inter_k T^{b(k)}_k)=0$, where $T_k^1:=T_k$ and $T_k^0:=T^{\text{co}}_k$.

\begin{lem}\label{banane}
Let $t\in [0,1]$ and let $\ms F=\{T_k:k\in\N\}$ be a non-atomic, $t$-effectively generated algebra. Let a function $\varphi$ from $C[0,1]$ to $[0,1]$ be defined by: $\varphi(\omega)=$ the unique member of $\cap\{\Phi(T_k):\omega\in T_k\}$, if it exists. \begin{enumerate}
\item[(1)] The domain of $\varphi$ includes all $t$-$\ms F$-randoms.

\item[(2)] If $\varphi(\omega)$ is defined then for each $k$,
$$\omega\in T_k\iff \varphi(\omega)\in\Phi(T_k).$$
\end{enumerate}
\end{lem}
\begin{proof}
(1): Suppose $\omega$ is not in the domain of $\varphi$. That is,  $S=\cap\{\Phi(T_k):\omega\in T_k\}$ does not have a unique element. It is clear that $S$ is an interval. Since $\ms F$ is non-atomic, this interval must have measure zero. Thus, since $S$ does not have exactly one element, $S$ must be empty.

By Remark \ref{referee} and Lemma \ref{245}, there is a $t$-computable point $a$ or $b$ such that $a_n\to a$ or $b_n\to b$, where $(a_n,b_n)=\inter_{k\le n}\Phi(T_k)$. Using this point $a$ or $b$ one can $t$-effectively determine whether $\omega\in T_k$, given any $k\in \N$. Thus $\omega$ is not $t$-$\ms F$-random.

(2)$\to$: By definition of $\varphi$.

(2)$\leftarrow$: Since $\{T_k\}_{k\in\N}$ is a Boolean algebra and so closed under complements, $$\omega\not\in T_k\to \omega\in T_k^{\text{co}}=T_{\ell}\to \varphi(\omega)\in\Phi(T_{\ell})=\Phi(T_k)^{\text{co}}.$$

\end{proof}

 \subsection{Effectiveness Lemmas}

A \emph{presentation} of a real number $a$ is a sequence of open intervals $I_n$ with rational endpoints, containing $a$, such that $I_n$ has diameter $\le 2^{-n}$.

\begin{lem}\label{less}
There is a Turing machine which, given a presentation of $a=a_0\oplus a_1$ as oracle, terminates iff $a_0<a_1$.
\end{lem}
\begin{proof}
Let $I_n$, $J_n$, be open intervals containing $a_0$, $a_1$, respectively, as in the definition of ``presentation''. Search for $n$, $m$ such that $I_n\inter J_m=\nil$, and the right endpoint of $I_n$ is $<$ the left endpoint of $J_m$. Such $n$, $m$ will be found if and only if $a_0<a_1$.
\end{proof}

\noindent On the other hand, it is well known that if $a_0=a_1$ then no algorithm will be able to verify this in general.
For intervals $(a,b)$, $(c,d)$, we say $(a,b)$ is \emph{bi-properly} contained in $(c,d)$ if $c<a\le b<d$.

\begin{rem}
The set of dyadic irrationals can be identified with a full-measure subset of $2^\N$ via the map $\iota$ such that $\iota(\sum_{i\ge 1} b_i 2^{-i})=\{b_i\}_{i\ge 1}$. This also gives an identification of cones $[\sigma]=\{A\in 2^\N: \forall n<|\sigma|\,\, A(n)=\sigma(n)\}$ for $\sigma\in \{0,1\}^*$ with intervals in the dyadic irrationals. Formally, we can let $\iota(2^\N)=[0,1]$ and if $\iota([\sigma])=[a,b]$ then $\iota([\sigma 0])=[a,a+(b-a)/2]$ and $\iota([\sigma 1]=[a+(b-a)/2,b]$. \end{rem}

\begin{lem}\label{makesyousmarter}
The set of pairs $\sigma\in\{0,1\}^*$, $k\in\N$ such that $\Phi(T_k)$ is bi-properly contained in $\iota([\sigma])$, is computably enumerable.
\end{lem}
\begin{proof} The endpoints of $\Phi(T_k)$ and $\iota([\sigma])$ have computable presentations. Thus the result follows from Lemma \ref{less}. \end{proof}

\begin{lem}\label{ML}
Let $t\in [0,1]$ and let $\ms F=\{T_k:k\in\N\}$ be a $t$-effectively generated algebra.
\begin{enumerate}
\item[(1)] If $\ms F$ is non-atomic, then for each $t$-ML-test $\{U_n\}_{n\in\N}$ there is a $t$-computable function $f:\N^2\to\N$ such that
$$U_n\cap\text{RAND}=\Union_m \Phi(T_{f(n,m)})\cap\text{RAND}.$$
\item[(2)] If for a $t$-computable function $f:\N^2\to\N$, we have $U_n=\Union_m \Phi(T_{f(n,m)})$ then $U_n$ has a subset $U_n'$ such that $U_n\cap\text{RAND}=U_n'\cap\text{RAND}$ and $\{U_n'\}_{n\in\N}$ is uniformly $\Sigma^0_1(t)$.
\end{enumerate}
\end{lem}
\begin{proof}
(1): We can enumerate the cones $[\sigma]$ contained in $U_n$. Once we see some $[\sigma]$ get enumerated and then see (using Lemma \ref{makesyousmarter}) that some $\Phi(T_k)$ is bi-properly contained in $[\sigma]$, we can enumerate $\Phi(T_k)$. Since $\ms F$ is non-atomic, we will gradually enumerate all of $[\sigma]$ except for possibly one or more of its endpoints. These endpoints are computable by Definition \ref{egs}(2).

(2): Let $\text{interior}(C)$ denote the interior of a set $C\subseteq [0,1]$. We let $$U_n=\bigcup_m \text{interior}(\Phi(T_{f(n,m)})).$$ Thus $U_n$ and $\bigcup_m\Phi(T_{f(n,m)})$ agree except on the left endpoints of the half-open intervals $\Phi(T_{f(n,m)})$. Since these endpoints are all computable numbers, we are done.

\end{proof}

The only result of this section that will be used in the next is the following:

\begin{thm}\label{yep}
Let $\omega\in\Omega$, $t\in [0,1]$, and let $\ms F_0=\{T_k:k\in\N\}$ be a non-atomic $t$-effective generating sequence, and $\ms F$ its generated algebra. The following are equivalent: \begin{enumerate}\item[(1)] $\omega$ is $t$-$\ms F$-random; \item[(2)] $\varphi(\omega)\in\text{RAND}^t$. \end{enumerate}
\end{thm}
\begin{proof} Let $\mu$ denote Lebesgue measure on $[0,1]$. 

\noindent (2) implies (1): Suppose $\omega$ is not $t$-$\ms F$-random, so $\omega\in\Inter_n V_n$, a $t$-$\ms F$-null set. Then $V_n=\Union_m T_{f(n,m)}$ for some $t$-computable $f$.

Let $U_n=\Union_m \Phi(T_{f(n,m)})$. Note

\noindent (a) $U_n$ is uniformly $\Sigma^0_1(t)$ by Lemma \ref{ML}(2).

\noindent (b) Since $\Phi$ is measure preserving on $\ms F$ and is a Boolean algebra homomorphism by Theorem \ref{XYZ},
$$\mu(\cup_{i=1}^n \Phi(T_{a_i}))=\mu(\Phi(\cup_{i=1}^n T_{a_i}))=W(\cup_{i=1}^n T_{a_i})$$
Since the measure of a countable union is the limit of the measures of finite unions, $\mu U_n=W(V_n)\le 2^{-n}$.

\noindent By (a) and (b), $\{U_n\}_{n\in\N}$ is a $t$-ML-test.

If $\omega$ is not in the domain of $\varphi$ then $\omega$ is not $\ms F$-random, by Lemma \ref{banane}(1); so we may assume $\varphi(\omega)$ exists. Hence, since $\omega\in\Inter_n V_n$, by definition of $\varphi$, we have $\varphi(\omega)\in\Inter_n U_n$.  Thus $\varphi(\omega)\not\in\text{RAND}^t$.

\noindent (1) implies (2): Suppose $\varphi(\omega)$ is not 1-$t$-random, so $\varphi(\omega)\in\Inter_n U_n$, for some $t$-Martin-L\"of test $\{U_n\}_{n\in\N}$. Let  $V_n:=\union_m T_{f(n,m)}$ with $f$ as in Lemma \ref{ML}(1). So by its definition, $V_n$ is uniformly $\Sigma^t_1(\ms F_0)$.
As in the proof that (2) implies (1), $\mu(U_n)=W(V_n)$. Since $\varphi(\omega)\in\Inter_n U_n$, by Lemma \ref{banane}(2) we have $\omega\in\Inter_n V_n$.
\end{proof}

\section{Khintchine's Law for Complex Oscillations}

It is common in probability theory to write, for $\omega\in\Omega$ and $x\in [0,1]$, $B_x(\omega)=\omega(x)$. This allows us to refer to the set $\{\omega\in\Omega: \omega(x)<y\}$, for example (where $x$, $y$ are fixed rational numbers) as the event that $B_x<y$, and as a set this is written $\{B_x<y\}$. In words, the value of the Brownian motion at time $x$ is less than $y$.

Let, for each $t\in [0,1]$, $\ms F_t$ be an $t$-effectively generated algebra that contains the one used in Theorem \ref{platformA}, and that moreover is non-atomic. The latter is achieved by including all events of the form $\{B_x<y\}$ for rational $x\in [0,1]$ and arbitrary rational $y$. Note that if $\ms F$ and $\ms F'$ are effectively generated algebras, and $\ms F\subseteq\ms F'$, then each $\ms F'$-random function $\omega\in\Omega$ is also $\ms F$-random, since adding elements to an effective generating sequence only adds new effective null sets.

\begin{lem}\label{schwimmbad}
For each $t\in [0,1]$ and each $\omega$ with $\varphi(\omega)\in\text{RAND}^t$, we have $LIL(\omega,t)$. In particular, for each $t\in\text{RAND}$ and each $\omega$ with $\varphi(\omega)\in\text{RAND}^t$, we have $LIL(\omega,t)$.
\end{lem}
\begin{proof}
Suppose $t\in [0,1]$ and $\varphi(\omega)\in\text{RAND}^t$. By Theorem \ref{yep}, $\omega$ belongs to no $t$-effective $\ms F_t$-null set. Hence by Theorem \ref{platformA}, $LIL(\omega,t)$.
\end{proof}

The point now is that in the image of $\varphi$, we already know more of what is going on. Let $A\oplus B=\{2n:n\in A\}\cup\{2n+1:n\in B\}$, for reals $A$, $B$ (equivalently, $A, B\subseteq \N$).

\begin{thm}[van Lambalgen's Theorem]\label{vl}
Let $A$, $B$ be reals. The following are equivalent.
\begin{itemize}
\item $A\in\text{RAND}$ and $B\in\text{RAND}^A$;
\item $A\oplus B\in\text{RAND}$;
\item $B\in\text{RAND}$ and $A\in\text{RAND}^B$.
\end{itemize}
\end{thm}

We can now approach our desired result:

\begin{lem}\label{soon}
If $\varphi(\omega)\in\text{RAND}$ and $t\in\text{RAND}^{\varphi(\omega)}$ then $LIL(\omega,t)$.
\end{lem}
\begin{proof}
Suppose $\varphi(\omega)\in\text{RAND}$ and $t\in\text{RAND}^{\varphi(\omega)}$. By Theorem \ref{vl} with $A=\varphi(\omega)$ and $B=t$, we have that $t\in\text{RAND}$ and $\varphi(\omega)\in\text{RAND}^t$. Hence by Lemma \ref{schwimmbad}, we have $LIL(\omega,t)$.
\end{proof}

\begin{thm}
If $\omega$ is a complex oscillation, then for almost all $t$, $LIL(\omega, t)$.
\end{thm}
\begin{proof}
Suppose $\omega$ is a complex oscillation. By Theorem \ref{yep}, $\varphi(\omega)\in\text{RAND}$. By Lemma \ref{soon}, $LIL(\omega, t)$ holds for each $t\in\text{RAND}^{\varphi(\omega)}$. Since $\text{RAND}^A$ has measure 1 for each $A\in 2^\N$, we are done.
\end{proof}

\noindent We remark that our main result can be extended from Martin-L\"of randomness to Schnorr randomness. To prove this one would use a weak version of van Lambalgen's theorem that holds for Schnorr randomness; see Merkle et al \cite{MMNRS}) or Yu \cite{Y}.

\section{Points of dimension $<1$}

We now show that almost surely, there are points of effective Hausdorff dimension $<1$ on the graph of 1-dimensional Brownian motion, other than the trivial example $(0,0)$. The question whether this is so was raised by J. S. Miller (personal communication) in connection with a more general question: Does there exists a continuous planar curve all of whose points have effective dimension exactly $1$? S. Lempp and J. Lutz have announced a proof that such a curve could not be a straight line. Although our result deals with a notion of effectivity, it is also a question about almost sure behavior and in that sense classical probability theory.

From now on we will denote the Wiener probability measure $W$ by $\P$, to facilitate probabilistic thinking. The sample path of 1-dimensional Brownian motion has value $B_t=B_t(\omega)\in\mathbb R$ at time $t\in [0,\infty)$, where $\omega$ is a randomly chosen member of the sample space of continuous functions $C(\mathbb R)$.

For $E\subseteq [0,\infty)$, let $B[E]=\{B_t: t\in E\}$ be the range of $B_t$ on $E$. The following two results are well-known, see for example \cite{Durrett}.

\begin{thm}[Blumenthal's 0-1 Law]
If $C$ is a property of Brownian motion that only depends on the germ at time $t=0$ of the Brownian path (that is, only depends on values for t arbitrarily close to 0) then $\P(C)\in\{0,1\}$.
\end{thm}

\begin{pro}\label{5}
$\frac{1}{\sqrt{a}}B_{at}$ and $B_t$ are identically distributed for $a>0$.
\end{pro}

\begin{thm}\label{0}[from Theorem 16.5 of \cite{Kahane}]
Let $E_1$ and $E_2$ be disjoint closed subsets of $[0,\infty)$. If
$\dim (E_1\times E_2)>1/2$, then $\P\{ B[E_1]\inter B[E_2]\ne\nil \}>0$.
\end{thm}

Theorem \ref{0} cannot be strengthened to probability one; for a counterexample we can take $E_2=\{0\}$ and $E_1=[a,b]$, where $0<a<b$.

\begin{df}\label{D}
For $\alpha\in [0,1]$ and $R\in 2^\N$, let  $D^R_\alpha=\{x: \forall c \exists n\, K^R(x\restrict n)<\alpha n-c\}$, where $K^R$ denotes prefix-free Kolmogorov complexity relative to $R$ (see \cite{LV}). Let $D_\alpha=D^\nil_\alpha$ and $D=D_{3/4}$. The $\alpha$-dimensional Hausdorff measure is denoted by $\mc H^\alpha$.
\end{df}

\begin{lem}\label{bordtennis}
For any $\mf X\subseteq 2^\N$, $\mc H^\alpha(\mf X)=0$ iff $\exists R$, $\mf X\subseteq D^R_\alpha$.
\end{lem}
\begin{proof}
Let $\mf X\subseteq 2^\N$ and let $\mc H^h(\mf X)$ be the $h$-dimensional Hausdorff measure of $\mf X$, where $h:\mathbb R\to\mathbb R$. By Theorem 1.14 of \cite{Reimann}, $\mc H^h(\mf X)=0$ iff $\exists R\in 2^\N\,\,\forall A\in\mf X\,\,\forall c\in\N\,\,\exists n$ 
$$2^{-K^R(A\restrict n)}\ge c h(2^{-n}),$$ i.e. 
$$K^R(A\restrict n)\le -\log h(2^{-n})-\log c.$$

We can replace $\log c$ by $c$. Hence taking $h(t)=t^\alpha$, this says 
$$\exists R\in 2^\N\,\,\forall A\in\mf X\,\,\forall c\in\N\,\,\exists n\,\, 
K^R(A\restrict n)\le \alpha n-c,$$
or
$$\exists R\in 2^\N\,\,\mf X\subseteq D^R_\alpha.$$
\end{proof}

\begin{lem}\label{kingston}
$\alpha\le\beta\Iff \exists R\,\,D_\alpha\subseteq D^R_\beta$.
\end{lem}
\begin{proof}
If $\alpha\le\beta$ then we can take $R=\nil$ and the inclusion is trivial.
If $\beta<\alpha$, $\beta\in\mathbb Q$, then this is not the case. Indeed, consider a join $B=A\oplus_Z\nil$ where $A$ is $R$-random and $Z$ is chosen to have density equal to the rational number $1-\beta=p/q$. That is $Z=\{n: n\mod q<p,\,\, n\in\N\}$. Here $A\oplus_Z B$ is such that the bits in $Z$ look like $B$, the others like $A$. For example, the usual $A\oplus B$ from computability theory is $A\oplus_{\{2n+1:n\in\N\}}B$. Then $B\in D_\alpha\backslash D^R_\beta$. Namely, $K^R(A\restrict n)\ge^+ n$ and so $K^R(B\restrict n)\ge^+\beta n$, or else we could describe $A\restrict n$ by describing $B\restrict (n/\beta)$ and then chopping off zeroes, giving $\exists d\forall c\exists n$, $K^R(A\restrict n)\le K^R(B\restrict n/\beta)+d\le \beta(n/\beta)-c+d$ .
\end{proof}

The \emph{effective Hausdorff dimension} of a given $x\in 2^\mathbb N$ is a notion that we need only indirectly. For completeness, we point out that it is defined to be the supremum of those $s\in [0,1]$ such that $x$ belongs to no set $\inter_n U_n$, where each $U_n$ is a $\Sigma^0_1$ class, uniformly in $n$, and $U_n=\bigcup_{p\in\N} [\sigma_{n,p}]$, $\sum_p 2^{-|\sigma| s}\le 2^{-n}$.

\begin{lem}\label{1}
\begin{enumerate}
\item[(a)] $\dim D=3/4$.
\item[(b)] If $x\in D$ then $x$ has effective Hausdorff dimension $\le 3/4$.
\end{enumerate}
\end{lem}
\begin{proof}
(a) By Lemma \ref{bordtennis}, $\mc H^\alpha(D_\alpha)=0$ and hence $\dim D_\alpha\le \alpha$, and by Lemmas \ref{bordtennis} and \ref{kingston}, $\dim D_\alpha\ge\alpha$. Part (b) follows from Theorem 2.6 of \cite{Reimann}.
\end{proof}

To show that $D$ has suitable closed subsets of large dimension, we describe and then use the \emph{potential theoretic method} \cite{Frostman35}\cite{MP}. We include some proofs from \cite{MP} for completeness, and to bring these ideas closer to a computability theoretical audience.

A measure $\mu$ on the Borel sets of a metric space $E$ is called a \emph{mass distribution} if $0<\mu(E)<\infty$.
Let the ultrametric $\upsilon$ be defined by $\upsilon(x,y)=2^{-\min\{n:x(n)\ne y(n)\}}$. 

\begin{df}
Suppose $\mu$ is a mass distribution on a metric space $(E,\rho)$ and $\alpha\ge 0$. The $\alpha$-potential of a point $x\in E$ with respect to $\mu$ is defined as
$$\phi_\alpha(x)=\int\frac{d\mu(y)}{\rho(x,y)^\alpha}.$$
The $\alpha$-energy of $\mu$ is 
$$I_\alpha(\mu)=\int\phi_\alpha(x)d\mu(x)=\iint \frac{d\mu(x)d\mu(y)}{\rho(x,y)^\alpha}.$$
\end{df}

Suppose $\mu$ is a mass distribution on $2^\N$, and suppose $\alpha\ge 0$. Then, for every $x\in 2^\N$, let $\mc B(x,r)$ be the closed ball centered in $x$ of radius $r$ and define the value $$\overline d_\alpha(\mu,x)=\limsup_{r\downarrow 0}\frac{\mu(\mc B(x,r))}{r^\alpha},$$
the upper $\alpha$-density of $\mu$ at $x$.

\begin{pro}[Local mass distribution principle]
If $\mu$ is a mass distribution on $2^\N$, and $A\subseteq 2^\N$ is a Borel set with $$\overline d_\alpha(\mu,x)<C\text{ for all }x\in A,$$ then $\mc H^\alpha(A)\ge\frac{\mu(A)}{C}$, and, in particular, if $\mu(A)>0$ then $\dim A\ge\alpha$.
\end{pro}
\begin{proof}
We first claim that $x\mapsto \mu(\mc B(x,r))$ is continuous for any $r$. Indeed pick $n$ such that $2^{-(n+1)}\le r< 2^{-n}$. Then $\mc B(x,r)=[x\restrict (n+1)]$ for any $x$. Thus if $[y\restrict n+1]=[x\restrict n+1]$ then $\mu\mc B(x,r)=\mu\mc B(y,r)$.

Now if $\delta>0$, let
$$A_\delta=\left\{x\in A: \forall r\in (0,\delta)\,\,\mu\mc B(x,r)\le Cr^\alpha\right\}.$$
Then $A_\delta$ is a Borel set, in fact closed if $A$ is closed.

We claim that $\mu(A_\delta)\le C\mc H^\alpha_\delta(A)$.

Indeed, letting $\text{diam}$ denote diameter induced by the standard metric on $[0,1]$, $\mc H^\alpha_\delta$ is the infimum of all sums $\sum_{i\in\N}\text{diam}([\sigma_i])^\alpha$ where $A\subseteq \bigcup_{i\in\N} [\sigma_i]$ and each $\text{diam}([\sigma_i])<\delta$.

Let $\sigma'_i$ be the subsequence of the $\sigma_i$ chosen so that $\mu [\sigma_i']\le C \text{diam}([\sigma'_i])^\alpha$. This may no longer cover $A$, but it covers $A_\delta$. Thus $\mu(A_\delta)\le
\sum_{i\in\N}\mu([\sigma'_i])\le \sum_{i\in\N}C\text{diam}([\sigma'_i])^\alpha \le C\sum_{i\in\N}\text{diam}([\sigma_i])^\alpha $. Since this is true for an arbitrary $\delta$-cover of $A$, we are done.

Now $\overline d_\alpha(\mu,x)<C$ for all $x\in A$ which means that
$$\limsup_{r\downarrow 0}\frac{\mu(\mc B(x,r))}{r^\alpha}<C$$
i.e.
$$\exists\eps>0\,\,\exists \delta\,\,\forall r\in (0,\delta)\,\,\frac{\mu(\mc B(x,r))}{r^\alpha}\le C-\eps$$

$$\exists\eps>0\,\,\exists \delta\,\,\forall r\in (0,\delta)\,\,\mu(\mc B(x,r))\le (C-\eps)r^\alpha < Cr^\alpha$$

which implies

$$\exists \delta\,\,\forall r\in (0,\delta)\,\,\mu(\mc B(x,r)) < Cr^\alpha$$
and
$$\exists \delta\,\,\forall r\in (0,\delta)\,\,\mu(\mc B(x,r)) \le Cr^\alpha$$

i.e. $\exists\delta\,\,x\in A_\delta$. So we have shown $A\subseteq \bigcup_{\delta>0}A_\delta$.

Since $\delta\le \delta' \Implies A_\delta\supseteq A_{\delta'}$,
$$\mu(A)\le \mu\bigcup_{\delta>~0}A_\delta = \lim_{\delta\downarrow 0}\mu A_\delta \le\lim_{\delta\downarrow 0} C\mc H^\alpha_\delta(A) = C \mc H^\alpha(A).$$
\end{proof}

\begin{thm}[Potential Theoretic Method]\label{PTM}
Let $\alpha\ge 0$ and let $\mu$ be a mass distribution on a Borel set $E\subseteq 2^\N$ with $I_\alpha(\mu)<\infty$. Then $\mc H^\alpha(E)=\infty$ and hence $\dim E\ge\alpha$.
\end{thm}
\begin{proof}
Note that since $I_\alpha(\mu)<\infty$, we have $\mu\{x\}=0$ for all $x\in E$. Let
$$E_1=\left\{x\in E: \overline d_\alpha(\mu,x)>0\right\}$$
 $$=\left\{x\in E: \limsup_{r\downarrow 0}\frac{\mu(\mc B(x,r))}{r^\alpha} >0\right\}$$
$$=\left\{ x\in E: \exists\eps>0\,\,\forall\eta\in (0,\eps)\forall\delta>0\exists r\in (0,\delta) \frac{\mu(\mc B(x,r))}{r^\alpha} \ge \eps-\eta\right\}.$$
This is the intersection of $E$ with a $\mathbf\Sigma^0_4$ set, hence Borel. By taking $\eta=\eps/2$ and then replacing $\eps/2$ by $\eps$, we see that
$$E_1\subseteq \left\{ x\in E: \exists\eps>0\,\,\forall\delta>0\exists r\in (0,\delta) \frac{\mu(\mc B(x,r))}{r^\alpha} \ge \eps\right\},$$
in fact these sets are equal. Thus by ``Skolemizing'',
$$E_1= \left\{ x\in E: \exists\eps>0\,\,\exists \{r_i\downarrow 0\}_{i\in\N}\,\, \frac{\mu(\mc B(x,r_i))}{r_i^\alpha} \ge \eps\right\}.$$

Now $\mu\{x\}=0$ and $\{x\}=\inter_{n\in\N} \mc B(x,2^{-n})$, so $\mu\mc B(x,2^{-n})\downarrow 0$. Hence a sufficiently much smaller ball around $x$ will have at most $3/4$ of a larger one's $\mu$-measure. In other words, there exist $0<q_i<r_i$, $B_i:=\mc B(x,r_i)\backslash\mc B(x,q_i)$, with $\mu B_i\ge \mu\mc B(x,r_i)/4 \ge \eps r_i^\alpha/4$.

We can arrange that $r_{i+1}<q_i$ by alternately choosing $r_i$, $q_i$, $r_{i+1}$, $q_{i+1}$. The annulus $B_i$ corresponds to the interval $(q_i,r_i]$ and hence the annuli are then pairwise disjoint. If $y\in B_i$ then $\upsilon(x,y)\le r_i$ so $\frac{1}{\upsilon(x,y)^\alpha}\ge r_i^{-\alpha}$. So we have

$$\phi_\alpha(x)=\int\frac{d\mu(y)}{\upsilon(x,y)^\alpha} \ge \sum_{i=1}^\infty \int_{B_i}\frac{d\mu(y)}{\upsilon(x,y)^\alpha} \ge \frac{\eps}{4}\sum_{i=1}^\infty r_i^\alpha r_i^{-\alpha} = \infty$$
whenever $x\in E_1$. But by assumption $I_\alpha(\mu)=\int\phi_\alpha(x)d\mu(x)<\infty$, so the only possibility is that $\mu(E_1)=0$. On the other hand, if $x\in E\backslash E_1$ then $\overline d_\alpha(\mu,x)=0$ and so $\forall C>0$, $\overline d_\alpha(\mu,x)<C$ which means that the Local Mass Distribution Principle applies. Since $E\supseteq E\backslash E_1$,

$$\mc H^\alpha(E)\ge \mc H^\alpha(E\backslash E_1)\ge C^{-1}\mu(E\backslash E_1)=C^{-1}\mu(E)$$
which means $\mc H^\alpha(E)=\infty$.

\end{proof}

Given any set $Z\subseteq\N$ we can form the tree $$T_Z=\{\sigma: (\forall n<|\sigma|)(Z(n)=0\to\sigma(n)=0\}.$$ 
For example, $T_{\N}=\{0,1\}^*$ (the set of all finite binary strings) and $T_\nil = \{0^n: n\in\N\}$.

\begin{lem}\label{stronger}
Suppose given a real number $\gamma\in (0,1)$, and $\eps>0$ such that $\gamma+\eps\in\mathbb Q$; say $\gamma+\eps=p/q$, $p,q\in\N$. If $A=[T_Z]$ with $Z=\{n: n\mod q<p\}$, then there is a probability measure $\mu$ on $A$ such that $I_\gamma(\mu)<\infty$.
\end{lem}
\begin{proof}
Let $\mu$ distribute the weight 1 on $A$ in the natural way, i.e. splitting the measure in half at each branching of $T_Z$. 
Fix $x$. Then $\mu\{y: \upsilon(x,y)=2^{-m /(\gamma+\eps)}\}=2^{-m}$. Writing $\alpha=\frac{\eps}{\gamma+\eps}$, 
 $$\int\frac{d\mu(y)}{\upsilon(x,y)^\gamma}=\sum 2^{-m} 2^{m\gamma/(\gamma+\eps)}$$
 $$= \sum 2^{-\left(\frac{\eps}{\gamma+\eps}\right)m}=\frac{1}{1-2^{-\alpha}}=\beta,$$
where $\beta$ is independent of $x$, and hence
$I_\gamma(\mu)=\int \beta d\mu(x) = \beta <\infty$.
\end{proof}

\begin{lem}\label{22}
Fix $1>\eps>0$. $D\inter [\eps,1]$ has a closed subset of dimension $\ge 2/3$.
\end{lem}
\begin{proof}
Let $Z=\{n:n\mod 3<2\}$. Then $[T_Z]\subseteq D$, as is easily seen (we can predict every third bit of any path in $[T_Z]$.) By the proof of Lemma \ref{stronger} and by Theorem \ref{PTM}, the dimension of $[T_Z]\inter [\sigma]$ is $\ge 2/3$ whenever $\sigma\in T_Z$; choosing $\sigma\in T_Z$ with $[\sigma]\subseteq [\eps,1]$, we are done.
\end{proof}

\section{Brownian motion}
\begin{lem}\label{bb}
Let $\mc Z=\mc Z(\omega)=\{t:B_t(\omega)=0\}$ be the set of zeroes of a path of Brownian motion $\omega$.
Let $X$ be any set of reals such that $X/2=\{x/2:x\in X\}\subseteq X$. Then
for any $n$, $$\P\{\mc Z\inter [2^{-(n+1)},2^{-n}]\inter X\ne\varnothing\} \ge \P\{\mc Z\inter [2^{-n},2^{-(n-1)}]\inter X\ne\nil\}.$$
\end{lem}
\begin{proof}
$$\P\{\mc Z\inter [2^{-n},2^{-(n-1)}]\inter X\ne\varnothing\}=$$
$$\P \left\{ \exists t\in \left[2^{-n},2^{-(n-1)}\right]\inter X,\,\, B_t=0\right\}=$$
$$\P \left\{ \exists s\in \left[2^{-(n+1)},2^{-n}\right]\inter \frac{X}{2},\,\, B_{2s}=0\right\}\le$$
$$\P \left\{ \exists s\in \left[2^{-(n+1)},2^{-n}\right]\inter X,\,\, \frac{1}{\sqrt{2}}B_{2s}=0\right\}=^{\text{Prop.\ref{5}}}$$
$$\P \left\{ \exists s\in \left[2^{-(n+1)},2^{-n}\right]\inter X,\,\, B_{s}=0\right\}=$$
$$\P\{\mc Z\inter [2^{-(n+1)},2^{-n}]\inter X\ne\nil\}.$$
\end{proof}

\begin{pro}\label{3}
Almost surely, there are points other than $(0,0)$ of
effective Hausdorff dimension $<1$ on the graph of Brownian motion $G_B=\{(t,B_t):t\ge 0\}$.
\end{pro}
\begin{proof} Let $E_1$ be a closed subset of $D$ of dimension $\ge 2/3$ as in Lemma \ref{22}, and let $E_2=\{0\}$. Note $E_1\inter E_2=\nil$ and $\dim E_1\times E_2\ge 2/3>1/2$. Hence by  Theorem \ref{0},

$$\P\{B[E_1]\inter B[E_2]\ne\nil\}>0.$$

By definition of Brownian motion, $B_0=0$ almost surely, so $B[E_2]=\{0\}$ and $B[E_1]\inter B[E_2]\ne\nil\Iff \mc Z\inter E_1\ne\nil$.
Since $E_1\subseteq D$ ($D$ as in Definition \ref{D}),
$\P\{\mc Z\inter D\ne\nil\}\ge\P\{\mc Z\inter E_1\ne\nil\}>0$.
By countable additivity there exists $n_0$ such that $\P\{\mc Z\inter D\inter [2^{-n_0},2^{-(n_0-1)}]\ne\nil\}>0$.

Clearly $D/2\subseteq D$. Hence by Lemma \ref{bb}, for any $n\ge n_0-1$ we have
$$0< \P\{\mc Z\inter [2^{-n_0},2^{-(n_0-1)}]\inter D\ne\nil\}$$
$$\le \P\{\mc Z\inter [2^{-(n+1)},2^{-n}]\inter D\ne\varnothing\} $$
$$\le \P\left\{\mc Z\inter \left[0,2^{-n}\right]\inter D\ne\varnothing\right\}$$

Let $C=\{\omega: \mc Z\inter D\inter [0,2^{-n}]\ne\nil, \forall n\ge 1\}$. Then
$$\P(C)=\lim_{n\to\infty} \P\{\mc Z\inter [0,2^{-n}]\inter D\ne\varnothing\}>0.$$

By Blumenthal's 0-1 Law, $\P(C)=1$ and so $\P\{\mc Z\inter D\ne\nil\}=1$.
If $t\in\mc Z\inter D$ then by Lemma \ref{1}(b), $(t,0)\in G_B$ has effective Hausdorff dimension $\le 3/4<1$, and we are done.
\end{proof}

\end{document}